\documentclass[a4paper,12pt]{amsart}
\usepackage{amssymb}
\usepackage{mathrsfs}
\usepackage[all]{xy}
\usepackage{color}
\usepackage{fancybox}
\usepackage[pdftex]{graphicx}

\theoremstyle{plain}
\newtheorem{theorem}{Theorem}[section]
\newtheorem{prop}[theorem]{Proposition}

\newtheorem{lemma}[theorem]{Lemma}

\theoremstyle{definition}
\newtheorem{remark}[theorem]{Remark}

\newtheorem{defn}[theorem]{Definition}
\newtheorem{example}[theorem]{Example}

\newtheorem{prob}[theorem]{Problem}

\renewcommand{\tilde}{\widetilde}
\renewcommand{\setminus}{\smallsetminus}
\newcommand{\R}{\mathbb R}

\newcommand{\Z}{\mathbb Z}
\newcommand{\C}{\mathcal B}
\newcommand{\Cc}{\mathbb C}
\newcommand{\D}{\mathcal D}
\newcommand{\Oo}{\mathcal O}
\newcommand{\vol}{\rm vol}
\newcommand{\diam}{\rm diam}
\newcommand{\dist}{\rm dist}

\numberwithin{equation}{section}

\title[]{A metric on the moduli space of bodies}
\date{}
\author[H.~Fujita, K.~Ohashi]{Hajime Fujita, Kaho Ohashi}

\address[H.~Fujita]{Department of Mathematical and Physical Sciences, Japan Women's University, 2-8-1 Mejirodai, Bunkyo-ku Tokyo, 112-8681, Japan}
\email{fujitah@fc.jwu.ac.jp}

\begin{document}

\maketitle

\begin{abstract}
We construct a metric on the moduli space of bodies in Euclidean space. The moduli space is defined as the quotient space  with respect to the action of integral affine transformations. This moduli space contains a subspace, the moduli space of Delzant polytopes, which can be identified with the moduli space of symplectic toric manifolds. We also discuss related problems. 
\end{abstract}

%\tableofcontents
%%%%%%%%%%%%%%%%%%%%%%%%%%%%%%%%
\section{Introduction}

%Gelfand-Cetlin�������邺. Newton-Okunkov body�������邺. 

An {\it $n$-dimensional Delzant polytope} is a convex polytope in $\R^n$ which is simple, rational and smooth\footnote{A vertex of a convex polytope in $\R^n$ is called {\it smooth} if directional vectors of edges at the vertex can be chosen as a basis of $\Z^n$. } at each vertex. There exists a natural bijective correspondence between the set of $n$-dimensional Delzant polytopes and the equivariant isomorphism classes of $2n$-dimensional {\it symplectic toric manifolds}, which is called the {\it Delzant construction} \cite{Delzant}. 
Motivated by this fact Pelayo-Pires-Ratiu-Sabatini studied in \cite{PPRS} the set of $n$-dimensional Delzant polytopes\footnote{In \cite{PPRS} they use the notation $\D_{\mathbb T}$ for the set of Delzant polytopes instead of $\D_n$. } $\D_n$ from the view point of metric geometry. They constructed a metric on $\D_n$ by using the $n$-dimensional volume of the symmetric difference.  They also studied the {\it moduli space of Delzant polytopes} $\tilde\D_n$, which is constructed as a quotient space with respect to natural action of integral affine transformations. It is known that $\tilde\D_n$ corresponds to the set of equivalence classes of symplectic toric manifolds with respect to {\it weak isomorphisms} \cite{KarshonKessler}, and they call the moduli space the {\it moduli space of toric manifolds}. 
In \cite{PPRS} they showed that, 
the metric space $\D_2$ is path connected, 
the moduli space $\tilde\D_n$ is neither complete nor locally compact. 
Here we use the metric topology on $\D_n$ and the quotient topology on $\tilde\D_n$. They also determined the completion of $\D_2$. 

Though in \cite{PPRS} they did not consider a metric on $\tilde \D_n$ it is natural to ask whether $\tilde\D_n$ has a natural metric or not. In this paper we give an answer for this question. Our main result is a construction of a metric on the moduli space $\tilde\D_n$.  We also show that its metric topology is homeomorphic to the quotient topology. Strictly speaking our construction does not rely on any conditions of the Delzant polytope so that it is possible to apply our proof for all {\it bodies}, that is, subsets in $\R^n$ obtained as the closure of bounded open subsets. Such a generalization for non-Delzant case has been increasing in importance in recent years in several areas. 
For example there exists an integrable system on a flag manifold, called the {\it Gelfand-Cetlin system} \cite{GuilleminSternbergGC}, which is studied from several context such as representation theory, algebraic geometry including mirror symmetry and so on (See \cite{HeeChoKim}, \cite{HaradaGC} or \cite{NishinouNoharaUeda} for example).  The Gelfand-Cetlin system is equipped with a torus action but it is not a symplectic toric manifold, however, it associates a non-simple convex polytope (as a base space of the integrable system), called the {\it Gelfand-Cetlin polytope}. More generally the theory of {\it Newton-Okounkov bodies} associated to divisors on varieties is developed recently, and it has great success in representation theory and algebraic geometry, especially for toric degeneration of varieties (See \cite{HaradaKaveh} , \cite{KavehKhovanskii} or \cite{LazarsfeldMustata} for example).

This paper is organized as follows. 
In Section~\ref{The moduli space of bodies} we introduce the moduli space of bodies with respect to the  group of integral affine transformations. 
In Section~\ref{Construction of the metric} we construct a metric $\tilde d$ on the moduli space following \cite{PPRS}. We also show that the metric topology and the quotient topology are homeomorphic to each other. 
In Section~\ref{Proof of the main theorem} we give a proof of our main theorem (Theorem~\ref{mainthm}), that is, $\tilde d$ is a metric on the moduli space.
In Section~\ref{Further problems} we propose several related problems. 

\section{The moduli space of bodies}\label{The moduli space of bodies}
Let $\C_n$ be the set of all bodies (i.e., compact subsets obtained as the closure of open subsets) in $\R^n$. 
Now we introduce the moduli space of bodies following \cite{PPRS}. 
Let $G_n:={\rm AGL}(n,\Z)$ be the integral affine transformation group. Namely $G_n={\rm GL}(n,\Z)\times \R^n$ as a set and the multiplication is defined by 
\[
(A_1, t_1)\cdot(A_2, t_2)=(A_1A_2, A_1t_2+t_1)
\]for each $(A_1, t_1), (A_2, t_2)\in G_n$. 
This group $G_n$ acts on $\C_n$ in a natural way, and $A\in \C_n$ and $B\in\C_n$ are called {\it $G_n$-congruent} if $A$ and $B$ are contained in the same $G_n$-orbit.   

\begin{defn}
The moduli space of the bodies in $\R^n$ with respect to the $G_n$-congruence $\tilde\C_n$ is defined by the quotient 
\[
\tilde\C_n:=\C_n/G_n. 
\]
\end{defn}

\begin{remark}\label{Delzant}
$\C_n$ contains an important subset, the set of {\it $n$-dimensional Delzant polytopes} $\D_n$. It is well known that there exists a bijective correspondence, the {\it Delzant construction}, between the set of all equivariant isomorphism classes of {\it $2n$-dimensional symplectic toric manifolds}. 

On the other hand, the quotient space $\tilde \D_n:=\D_n/G_n$ corresponds to the {\it weak equivalence classes} of $2n$-dimensional symplectic toric manifolds. It is called the {\it moduli space of toric manifolds} in \cite{PPRS}. 
\end{remark}

\section{Construction of the metric}\label{Construction of the metric}

Let $d:\C_n\times \C_n\to \R$ be the function defined by 
\[
d(A,B):={\vol}_n(A\bigtriangleup B)=\int_{\R^n} \chi_{A\bigtriangleup B}d\lambda=\int_{\R^n}|\chi_A-\chi_B|d\lambda
\]for each $A,B\in\C_n$, where $A\bigtriangleup B:=(A\setminus B)\cup (B\setminus A)$ is the symmetric difference of $A$ and $B$,  ${\vol}_n(X)$ denotes the $n$-dimensional volume of $X$ with respect to the $n$-dimensional Lebesgue measure $d\lambda$ and $\chi_X:\R^n\to\R$ denotes the characteristic function of a subset $X$ of $ \R^n$. Then one can see that $d$ is a metric on $\C_n$. 

Note that the function $\tilde\C_n\times\tilde\C_n\ni([P_1], [P_2])\mapsto d(P_1, P_2)$ is not well-defined. So we introduce the function by taking the infimum among the values of $d$. 

\begin{defn}\label{deftilded}
Define a function $\tilde d:\tilde\C_n\times\tilde\C_n\to \R$ by 
\[
\tilde d(\alpha, \beta):=\inf\{d(P_1, P_2) \ | \ [P_1]=\alpha, [P_2]=\beta\}
\] for $(\alpha, \beta)\in \tilde\C_n\times\tilde\C_n$. 
\end{defn}

For $g=(A,t)\in G_n$, since $\det A=\pm 1$ we have $d(gP_1,gP_2)=d(P_1, P_2)$ for any $P_1, P_2\in \C_n$. When we fix representatives $P_1\in \C_n$ of $\alpha\in \tilde\C_n$ and $P_2\in \C_n$ of $\beta\in\tilde\C_n$ it follows that 
\[
\tilde d(\alpha, \beta)=\inf\{d(g_1P_1, g_2P_2) \ | \ g_1, g_2\in G_n\}=\inf\{d(P_1, gP_2) \ | \ g\in G_n\}. 
\]

Clearly $\tilde d$ is a symmetric function. The following is the main theorem of the present paper. 

\begin{theorem}\label{mainthm}
$\tilde d$ is a metric on $\tilde \C_n$. 
\end{theorem}

We show this theorem by proving the triangle inequality in Proposition~\ref{triineq} and the positiveness\footnote{It is known that the function defined as in Definition~\ref{deftilded} is a  pseudo-metric and it satisfies the positiveness if and only if each $G_n$-orbit is a closed subset in $\C_n$. See \cite[p.80]{Shioyabook} for example.  In fact our proof of Proposition~\ref{positiveness} is a proof to show the closedness of the orbit.} in Proposition~\ref{positiveness}. 

Due to Theorem~\ref{mainthm},  we can consider the metric topology $\Oo_{\tilde d}$ of $\tilde\C_n$. On the other hand by using the metric topology of $\C_n$ and the natural projection $\pi:\C_n\to\tilde\C_n$,  we also have the quotient topology $\Oo_{\pi}$ of $\tilde\C_n$. We have the following. 

\begin{theorem}
Two topological spaces $(\tilde\C_n,\Oo_{\tilde d})$ and $(\tilde\C_n, \Oo_{\pi})$ are homeomorphic to each other. In particular $(\tilde\C_n, \Oo_\pi)$ is a Hausdorff space. 
\end{theorem}

This theorem follows from the following general proposition. 

\begin{prop}\label{prophomeo}
Suppose that $(X,d)$ is a metric space and a group $G$ acts on $X$ in an isometric way. Let $\tilde X:=X/G$ be the quotient space and assume that the function $\tilde d:\tilde X\times \tilde X\to \R$ defined by 
\[
\tilde d([x_1], [x_2])=\inf\{d(x_1, gx_2) \ | \ g\in G\} \quad (([x_1], [x_2])\in \tilde X\times \tilde X) 
\] is a metric on $\tilde X$. Then the quotient topology $\Oo_\pi$ and the metric topology $\Oo_{\tilde d}$ on $\tilde X$ are homeomorphic to each other. 
\end{prop}

\begin{proof}
Suppose that $U\in \Oo_{\pi}$ and take $\alpha\in U$. Fix a representative $x\in \pi^{-1}(\alpha)\subset\pi^{-1}(U)$. Since $\pi^{-1}(U)$ is an open subset in $X$ there exists $\delta>0$ such that $B_{\delta}(x)(\subset X)$, the open ball of radius $\delta$ centered at $x$, is contained in $\pi^{-1}(U)$. For this $\delta$ let $\tilde B_{\delta}(\alpha)\subset\tilde X$ be the open ball of radius $\delta$ centered at $\alpha$. For each $\beta\in\tilde B_{\delta}(\alpha)$ and a representative $x'\in\pi^{-1}(\beta)$ since 
\[
\tilde d(\alpha, \beta)=\inf_{g}\{d(x, gx')\}<\delta
\]there exists $g\in G$ such that $d(x,gx')<\delta$. It implies that $gx'\in B_{\delta}(x)\subset \pi^{-1}(U)$, and hence, $\beta=\pi(gx')\in U$. So we have $\tilde B_{\delta}(\alpha)\subset U$, and it means $U\in \Oo_{\tilde d}$. 

Conversely suppose that $U\in \Oo_{\tilde d}$ and take $x\in \pi^{-1}(U)$.  
%put $\alpha:=\pi(x)\in U$. 
Then there exists $\delta>0$  such that $\tilde B_{\delta}(\pi(x))\subset U$. For this $\delta>0$ consider the open ball $B_{\delta}(x)$ in $X$. Take $x'\in B_{\delta}(x)$ then we have 
$\tilde d(\pi(x), \pi(x'))\leq d(x, x')<\delta$, and hence, $\pi(x')\in \tilde B_{\delta}(\pi(x))\subset U$. It implies that $x'\in \pi^{-1}(U)$. So we have $B_{\delta}(x)\subset \pi^{-1}(U)$, and it means $U\in \Oo_{\pi}$.  
\end{proof}

\section{Proof of the main theorem}\label{Proof of the main theorem}

\begin{prop}\label{triineq}
The function $\tilde d$ defined in Definition~\ref{deftilded} satisfies the triangle inequality. 
\end{prop}
\begin{proof}
For $\alpha, \beta,\gamma\in \tilde\C_n$ we take and fix their representatives $P_1, P_2, P_3\in\C_n$. By the triangle inequality of $d$ we have, 
\[
d(P_1, gP_2) \leq d(P_1, g'P_3)+d(g'P_3, gP_2)= d(P_1, g'P_3)+d(g^{-1}g'P_3, P_2)
\] for any $g,g'\in G_n$. By taking the infimum for $g$ we have 
\begin{eqnarray*}
\tilde d(\alpha, \beta)&\leq& d(P_1, g'P_3)+\inf_g\{d(g^{-1}g'P_3, P_2)\}\\
&=&d(P_1, g'P_3)+\inf_g\{d(g''g^{-1}g'P_3,g''P_2)\} \\ 
&\leq&d(P_1, g'P_3)+d(P_3, g''P_2)\}
\end{eqnarray*} for any $g''\in G_n$. By taking the infimum for $g', g''\in G_n$ we have the triangle inequality 
\[
\tilde d(\alpha,\beta)\leq \tilde d(\alpha,\gamma)+\tilde d(\gamma,\beta). 
\]
\end{proof}

\begin{remark}
The above proof also works for the same general situation in Proposition~\ref{prophomeo}. 
\end{remark}

To show the positiveness of $\tilde d$ we use the following two elementary lemmas. 

\begin{lemma}\label{keylemma2}
For $a>0$ let $C$ be a cube $C:=[-2a, 2a]^n=[-2a,2a]\times[-2a, 2a]\times\cdots\times[-2a, 2a]\subset \R^n$. Then for any affine hyperplane $H$ in $\R^n$ there exists a vertex $v$ of $C$ such that 
\[
{\dist}(v,H)>a. 
\]
\end{lemma}

\begin{proof}
Suppose that $H$ is an affine hyperplane with defining equation 
\[
f(x):=\nu\cdot x-2c=0 \quad (x\in\R^n), 
\]where $\cdot$ denotes the Euclidean inner product, $\nu=(\nu_1, \ldots, \nu_n)\in \R^n$ is a unit normal vector of $H$ and $c\in \R$. Let $v_i \ (i=1,2,\ldots, 2^n)$ be the vertices of $C$. Suppose that 
\[
{\dist}(v_i,H)\leq a  
\]for any $i=1,2,\ldots, 2^n$. Then for any $v_i$ we have 
\[
a\geq {\dist}(v_i,H)=|f(v_i)|=2\left|a\sum_{k=1}^n(\pm\nu_k)-c\right|. 
\]By taking the square we have 
\[
a^2\geq
4\left(a^2\left(1+2\sum_{k\neq l}(\pm \nu_k\nu_l)\right)-2ac\sum_{k=1}^n(\pm \nu_k)+c^2\right), 
\]
and hence, we have inequalities 
\[
8a^2\sum_{k\neq l}(\pm \nu_k\nu_l)-8ac\sum_{k=1}^n(\pm \nu_k)\leq -3a^2-4c^2<0
\]for all $i=1,\ldots, 2^n$. On the other hand the sum of the left hand of the inequality for all $i$ is equal to $0$, we have a contradiction 
\[
0\leq 2^n(-3a^2-4c^2)<0.
\]
So we have ${\rm dist}(v_i, H)>a$ for some $i$.
\end{proof}

\begin{lemma}\label{keylemma1}
Let $C$ be an $n$-dimensional cube in $\R^n$. For a sequence $\{g_m\}_m\subset G_n$ we put $C_m:=g_mC$. If ${\diam}(C_m)$ goes to $\infty$ as $m\to \infty$, then there exist a sequence $\{F_m\}_m$ of $n-1$-dimensional faces of $C_m$ such that
\[
{\dist}(H_m,H_m')\to 0 \ (m\to \infty), 
\] where $H_m$ is the $n-1$-dimensional affine subspace containing $F_m$ and $H_m'$ is the other affine hyperplane containing the other $n-1$-dimensional face in $C_m$ which is parallel to $F_m$. 
\end{lemma}

To show this lemma by induction we generalize it in the following way as in Figure~\ref{keylem1gen.png}.

\begin{lemma}\label{keylemma1gen}
Let $C$ be an $n$-dimensional cube in $\R^n$. For a sequence $\{g_m\}_m\subset G_n$ we put $C_m:=g_mC$. For a sequence $\{F_m\}_m$ of $k(\geq 2)$-dimensional faces of $C_m$, if the $k$-dimensional volumes (with respect to the $k$-dimensional Hausdorff measure) $\{{\vol}_k(F_m)\}_m$ is bounded and the diameters ${\diam}(F_m)$ goes to $\infty$ as $m\to \infty$, then there exist a sequence $\{E_m\}_m$ of $k-1$-dimensional faces of $F_m$ such that
\[
{\dist}(H_m,H_m')\to 0 \ (m\to \infty), 
\] where $H_m$ is the $k-1$-dimensional affine subspace containing $E_m$ and $H_m'$ is the other $k-1$-dimensional affine subspace containing the other $k-1$-dimensional face in $F_m$ which is parallel to $E_m$. 
\end{lemma}

\begin{figure}[h]
\begin{center}
\includegraphics[scale=0.45]{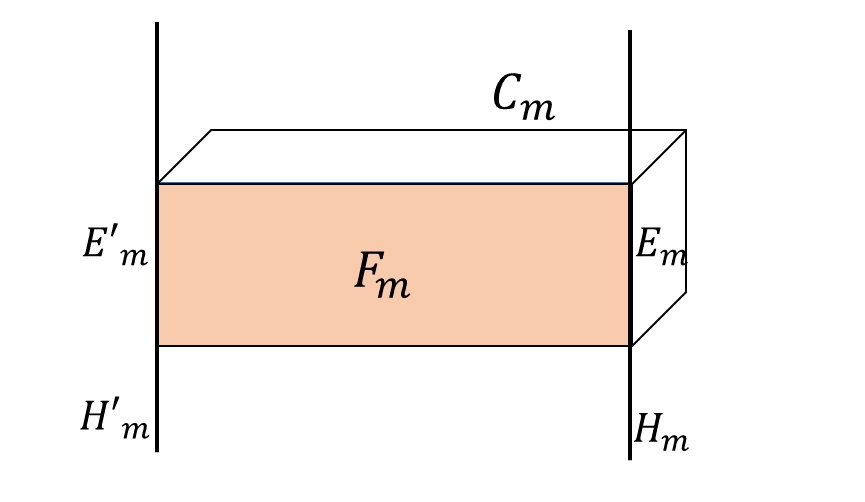}
\caption{} \label{keylem1gen.png}
\end{center}
\end{figure}

\begin{proof}
We show this lemma by induction on $k$. For $k=2$ since ${\diam}(F_m)\to\infty$ one can see by the triangle inequality that the length of the diagonal of $F_m$ goes to $\infty$, and hence, ${\vol}_1(E_m)\to \infty$ for some edge $E_m$ of $F_m$.  Since for the line $H_m$ containing $E_m$ and the other line $H_m'$ containing the edge which is parallel to $E_m$, we have 
\[
{\vol}_2(F_m)={\dist}(H_m,H_m'){\vol}_1(E_m)
\] and it is bounded. It implies that ${\dist}(H_m,H_m')\to 0$. 

Now we assume that the statement holds for any sequence of $k-1$-dimensional faces of $C_m$. Suppose that a sequence $\{F_m\}_m$ of $k$-dimensional faces of $C_m$ satisfies that ${\diam}(F_m)\to \infty$ and $\{{\vol}_k(F_m)\}_m$ is bounded. For any $k-1$-dimensional face $E_m$ of $F_m$ we have 
\begin{equation}\label{keylemma1gen1}
{\vol}_k(F_m)={\dist}(H_m, H_m'){\vol}_{k-1}(E_m). 
\end{equation}
If there exists a sequence $\{E_m\}_m$ with ${\vol}_{k-1}(E_m)\to \infty$, then we have ${\dist}(H_m, H_m')\to 0$, and hence the sequence $\{E_m\}_m$ is the required one. We assume that the $k-1$-dimensional volumes of any $k-1$-dimensional face of $F_m$ is bounded. The assumption ${\diam}(F_m)\to \infty$ and the triangle inequality imply that there exists a sequence of edges $\{e_m\}_m$ of $F_m$ with ${\vol}_1(e_m)\to \infty$. For any $k-1$-dimensional face $E_m$ containing $e_m$, the assumption of the induction implies that there exists sequences of $k-2$-dimensional affine subspaces $\{f_{1,m}\}_m$ and $\{f_{1,m}'\}_m$ containing $k-2$-dimensional face of $E_m$ such that 
\[
{\dist}(f_{1, m},f_{1, m}')\to 0 \quad (m\to \infty). 
\]
For the other $k-1$-dimensional face $E_m'$ of $F_m$ which is parallel to $E_m$, let $f_{2,m}$ and $f_{2,m}'$ be the affine subspaces which corresponds to $f_{1,m}$ and $f_{1,m}'$ by the translation which maps to $E_m$ to $E_m'$. Let $G_m$ be the $k-1$-dimensional face of $F_m$ containing $f_{1,m}$ and $f_{2,m}$.  Let $\hat H_m$ be the affine subspace containing $G_m$ and $\hat H_m'$ the other $k-1$-dimensional affine subspace containing the $k-1$-dimensional face of $F_m$ which is parallel to $G_m$. Then we have 
\[
{\dist}(\hat H_m, \hat H_m')={\dist}(f_{1,m}', f_{2,m}')={\rm dist}(f_{1,m}, f_{2,m})\to 0 \quad  (m\to \infty). 
\]It implies that this $G_m$ is the required one, and we complete the proof. 
\end{proof}

\begin{prop}\label{positiveness}
The function $\tilde d$ defined in Definition~\ref{deftilded} satisfies the positiveness. 
\end{prop}

\begin{proof}
Suppose that $\tilde d(\alpha,\beta)=0$ for $\alpha, \beta\in \tilde\C_n$. Take  representatives $P_1, P_2\in\C_n$ of $\alpha, \beta$ and a minimizing sequence $\{g_m\}_m\subset G_n$, i.e., 
\[
d(P_1,g_mP_2)\to \tilde d(\alpha,\beta)=0 \quad (m\to \infty). 
\]We show that $\{g_m\}_m$ is a bounded sequence in $G_n$ with respect to the direct product metric of the Euclidean metrics on $G_n={\rm GL}(n,\Z)\times \R^n\subset \R^{n^2}\times \R^n$. If $\{g_m\}_m$ is bounded, then by taking a limit $g\in G_n$ of a convergent subsequence of $\{g_m\}_m$  we have $d(P_1, gP_2)=0$, and hence, $\alpha=[P_1]=[gP_2]=\beta$. 

We may take representatives $P_1$ and $P_2$ of $\alpha$ and $\beta$ so that $P_1\cap P_2$ contains the origin of $\R^n$ as an interior point. For $a>0$ small enough we take an $n$-dimensional cube $C:=[-2a,2a]^n\subset P_1\subset \R^n$. We also take an $n$-dimensional large cube $C'$ so that $P_2\subset C'$. See Figure~\ref{posi1}. Consider a sequence of polytopes  $\{C_m:=g_mC'\}_m$.

\begin{figure}
\begin{center}
\includegraphics[scale=0.55]{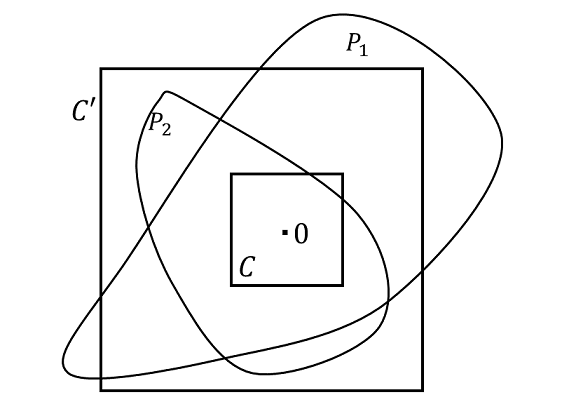}
\caption{} \label{posi1}
\end{center}
\end{figure}

Hereafter for any unbounded sequence we may consider it goes to $\infty$ by taking a subsequence. Suppose that $\{g_m=(A_m,t_m)\}_m$ is unbounded. Then the norm $|g_m|=|A_m|+|t_m|$ goes to $\infty$. 

\noindent
(i) Suppose that $|A_m|$ goes to $\infty$. We may assume that the $(1,1)$-entry of $A_m$ goes to $\infty$ and put $a_m$ the $(1,1)$-entry of $A_m$. Then for some $r>0$ small enough we have 
\[
{\diam}(C_m)\geq r|a_m|\to \infty \quad (m\to \infty). 
\]Since ${\diam}(C_m)\to \infty$ and ${\vol}(C_m)={\vol}(C')$ is bounded, by Lemma~\ref{keylemma1}, there exists an $n-1$dimensional face $F_m$ of $C_m$ such that 
\[
{\dist}(H_m,H_m')\to 0 \ (m\to \infty), 
\] where $H_m$ is the $n-1$-dimensional affine subspace containing $F_m$ and $H_m'$ is the $n-1$-dimensional affine subspace containing the $n-1$-dimensional face of $C_m$ which is parallel to $F_m$. In particular we have 
\begin{equation}\label{<a/2}
{\dist}(H_m, H_m')<\frac{a}{2} 
\end{equation} for $m\gg 1$. On the other hand by Lemma~\ref{keylemma2}, there exists a vertex $v_{m}$ of $C$ for each $H_m$ such that 
\begin{equation}\label{>a}
{\dist}(v_m, H_m)>a. 
\end{equation}
By (\ref{<a/2}) and (\ref{>a}) we have 
\[
{\dist}(v_m, C_m)\geq {\dist}(v_m,H_m)-\frac{a}{2}>\frac{a}{2}, 
\]and hence, 
\begin{equation}\label{empty}
B_{a/4}(v_m)\cap C \cap C_m=\emptyset, 
\end{equation}
where $B_{a/4}(v_m)$ is the open ball in $\R^n$ of radius $a/4$ centered at $v_m$.  See Figure~\ref{posi2}. 
Since $P_1\bigtriangleup g_mP_2\supset P_1\setminus g_mP_2\supset C\setminus C_m\supset B_{a/4}(v_m)\cap C$,  we have 
\[
d(P_1, g_mP_2)\geq {\vol}_n(C\setminus C_m)\geq{\vol}_n(B_{a/4}(v_m)\cap C)=\frac{1}{2^n}{\vol}(B_{a/4}(0))>0 
\] by (\ref{empty}). This contradicts to $d(P_1, g_mP_2)\to 0$. So we have that  $\{A_m\}_m$ is a bounded sequence\footnote{Since ${\rm GL}(n,\Z)$ is a discrete space without accumulation points, a convergent subsequence of $\{A_m\}_m$ in ${\rm GL}(n,\Z)$ is a constant sequence for $m\gg1$.}.

\begin{figure}
\begin{center}
\includegraphics[scale=0.55]{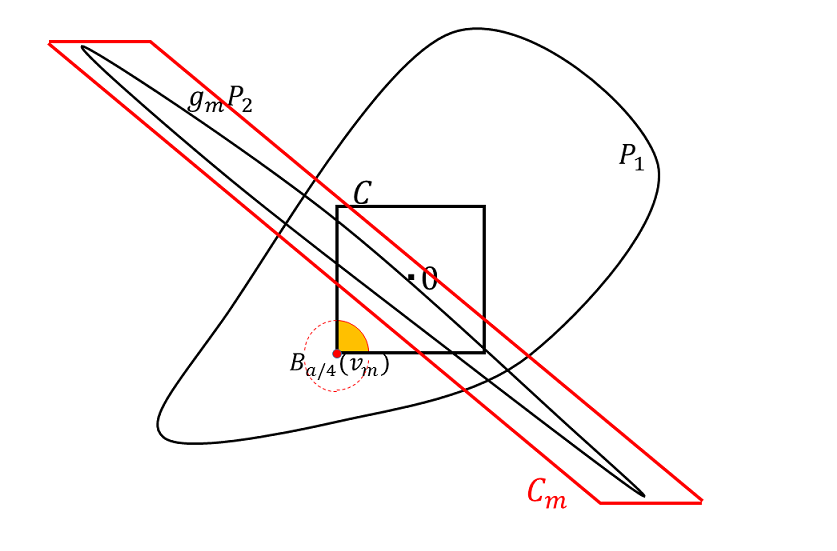}
\caption{} \label{posi2}
\end{center}
\end{figure}

\noindent
(ii) 
Suppose that $\{t_m\}_m$ is unbounded and $\{A_m\}_m$ is bounded. In this case since $r_m:={\diam}(C_m)$ is a bounded sequence, by fixing a point $p_m\in C_m$ we have 
\[
C_m\subset B_{r_m}(p_m)\subset B_R(p_m)
\]for some $R\gg 0$ and any $m$. Since $\{t_m\}_m$ is unbounded and $\{A_m\}_m$ is bounded, we have ${\dist}(p_m,C)\to \infty$. In particular we have ${\dist}(p_m,C)>2R$, $C\cap B_R(p_m)=\emptyset$, and hence, 
\[
C\cap C_m=\emptyset
\]for $m\gg 1$. It implies that $C\subset P_1\bigtriangleup g_m P_2$ and 
\[
d(P_1,g_mP_2)>{\vol}_n(C)>0. 
\]
This contradicts to $d(P_1, g_mP_2)\to 0$. So we also have $\{t_m\}_m$ is a bounded sequence. 

By (i) and (ii) we have that $\{g_m\}_m$ is a bounded sequence and complete the proof. 
\end{proof}

\section{Further problems}\label{Further problems}

In this section we focus on the moduli space of the Delzant polytopes $\tilde\D_n=\D_n/G_n$. As we noted in Remark~\ref{Delzant}, this moduli space can be identified with the set of all equivalence classes of symplectic toric manifolds with respect to the weak isomorphisms \cite{KarshonKessler}.

\subsection{Completion}
As it is shown in \cite{PPRS}, the metric space $(\D_n, d)$ is not complete. Namely there exists a Cauchy sequence\footnote{In \cite{PPRS} they constructed such a sequence in $\D_2$. By taking the product with cubes $[0,1]^{n-2}$ the sequence gives a sequence with the same property in $\D_n$. }  $\{T_m\}_m$ in $\D_n$ which does not converge in $\D_n$.  In \cite{PPRS} they determined the completion of $(\D_2, d)$. 
\begin{theorem}[\cite{PPRS}, Theorem~{8}]
The completion $\overline{(\D_2, d)}$ of $(\D_2, d)$ is isometric to 
\[
\left({\mathcal C}'_2/{{\sim}}\right)\cup \{\emptyset\}, 
\]where ${\mathcal C}'_2$ is the set of all compact convex subsets in $\R^2$ with positive Lebesgue measure and $A\sim B$ if ${\vol}_2(A\bigtriangleup B)=0$ for $A, B\in{\mathcal C}'_2$. 
\end{theorem}

The Cauchy sequence $\{T_m\}_m$ in $\D_n$ gives a Cauchy sequence in $\tilde\D_n$ which is not a convergent sequence\footnote{In fact it converges in $\tilde\C_n$.} in $\tilde\D_n$. In particular $(\tilde\D_n, \tilde d)$ is not complete. One may consider the following problem. 

\medskip

\begin{prob}
Determine the completion of $(\tilde\D_n, \tilde d)$. For instance whether the completion of ${\tilde\D_2}$ is equal to $\overline{\D_2}/G_2$ or not?
\end{prob}

\subsection{Dimension of the moduli space}
Each metric space has an important invariant, the {\it Hausdorff dimension}. Then one may ask the following.

\begin{prob} Estimate the Hausdorff dimension of $(\D_n, d)$ or $(\tilde \D_n, \tilde d)$.  
\end{prob}

\medskip

This may be a naive problem because even if we fix $n$, the dimension of polytopes, the number of vertices increase any number, and hence, the dimension may become infinity. One candidate to have a reasonable answer is to introduce a stratification of $\D_n$ or $\tilde\D_n$ by the numbers of vertices. 

\begin{defn}
For each natural number $l$ define $\D_n^{(l)}$ and $\tilde\D_n^{(l)}$ by 
\[
\D_n^{(l)}:=\{P\in \D_n \ | \ P \ {\rm has} \ l \ {\rm vertices.}\}
\]
and 
\[
\tilde\D_n^{(l)}:=\D_n^{(l)}/G_n. 
\]
\end{defn}

Note that $d$ and $\tilde d$ induce metrics on $\D_n^{(l)}$ and $\tilde\D_n^{(l)}$ respectively.

\begin{prob} Estimate the Hausdorff dimension of $(\D_n^{(l)}, d)$ or $(\tilde \D_n^{(l)}, \tilde d)$. 
\end{prob}

\subsection{Relation with Gromov-Hausdorff distance}\label{Relation with Gromov-Hausdorff distance}
For each Delzant polytope $P\in \D_n$, one can canonically associate a compact symplectic toric manifold $M_P$ by the Delzant construction procedure. In fact the symplectic manifold $M_P$ carries canonical K\"ahler structure, in particular it is equipped with a Riemannian metric. By the explicit description of the K\"ahler metric in terms of polytope due to Guiellemin \cite{GuilleminKahler}, as it is noted in \cite{Abreu}, one can see that if two Delzant polytopes $P$ and $P'$ are $G_n$-congruent each other, then two compact Riemannian manifolds $M_P$ and $M_{P'}$ are isometric. In this way we have a natural map between two metric spaces 
\[
{\rm Del} : \tilde\D_n\to {\mathcal M}, \quad [P] \mapsto {\rm Del}([P]) := M_P,  
\]where ${\mathcal M}$ is the metric space consisting of all (isometry classes of) compact Riemannian manifolds equipped with  the {\it  Gromov-Hausdorff distance} $d_{GH}$.  See \cite[Chapter~7]{BBI} for its definition. Then one may consider the following (vague) problem. 

\begin{prob}
Study the map Del from the both viewpoints of symplectic geometry and metric geometry. 
\end{prob}

\begin{remark}
It would be also interesting to study the map ${\rm Del}$ by considering ${\mathcal M}$ as a metric space by the {\it intrinsic flat metric} $d_F$, which is introduced in \cite{SormaniWenger}. 
\end{remark}

Unfortunately the moduli space $\tilde\D_n$ would not be suitable for this problem as seen in the following example. 

\begin{example}\label{excollapse}
Let $\{P_m\}_m$ be a sequence in $\D_2$ defined by rectangles $P_m:=[0,1]\times[0,1/m]$. It is known that the corresponding Riemannian manifolds by Del is a sequence $\{(S^2\times S^2, g_{FS}\oplus \frac{1}{m}g_{FS})\}_m$ in ${\mathcal M}$, where $g_{FS}$ is the Fubini-Study metric on ${\Cc}P^1=S^2$. Its limit with respect to the Gromov-Hausdorff distance is $(S^2, g_{FS})$, which is a 2-dimensional symplectic toric manifold constructed from the closed interval $[0,1]\in \D_1$ by the Delzant construction. On the other hand the limit of $\{P_m\}_m$ with respect to the metric $d$ (in the completion $\overline{\D_2}$) is the empty set $\emptyset$. See Figure~\ref{s2ex}. 
\end{example}

\begin{figure}
\begin{center}
\includegraphics[scale=0.55]{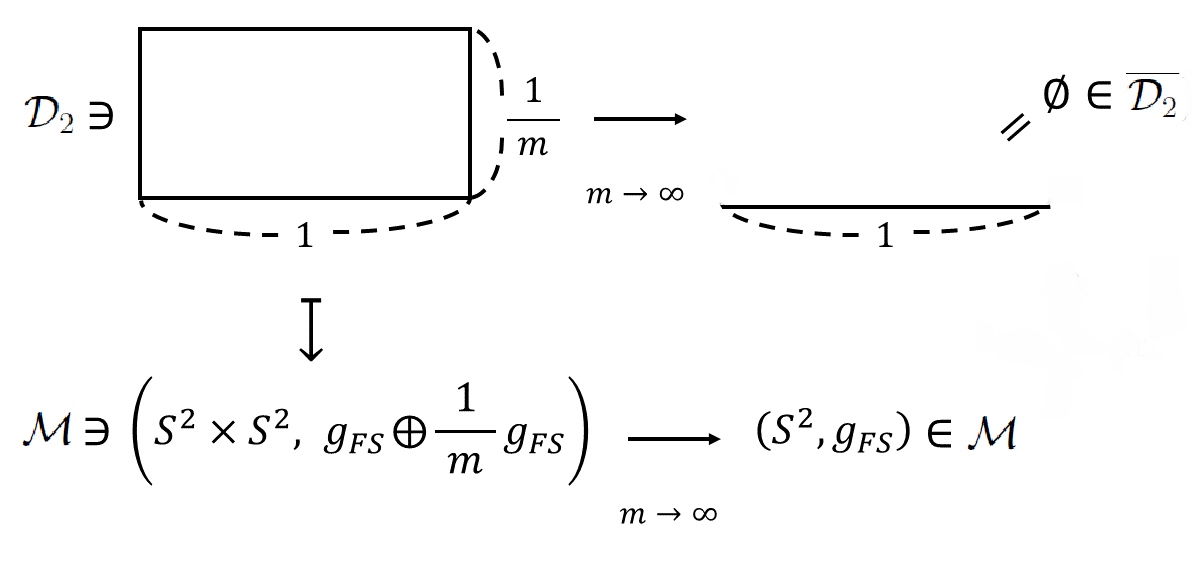}
\caption{} \label{s2ex}
\end{center}
\end{figure}

To overcome this problem we first extend the moduli space in the following way.  
\begin{defn}
For each non-negative integer $k$ with $0\leq k\leq n$ let $\iota_k:\R^k\to \R^n$ be the inclusion, $\R^k\ni(x_1,x_2,\ldots, x_k)\mapsto (x_1, x_2, \ldots, x_k, 0,\ldots, 0)\in\R^n$. We define the set $\D_{\leq n}$ by
\[
\D_{\leq n}:=\{g(\iota_k(P)) \ | \ P\in\D_k, \ g\in G_n\}, 
\]and $\tilde\D_{\leq n}$ by 
\[
\tilde\D_{\leq n}:=\D_{\leq n}/G_n. 
\]
\end{defn}

Now we extend the map ${\rm Del}$ to a map from $\tilde\D_{\leq n}$. For given equivalence class $[g(\iota_k(P))]\in\tilde\D_{\leq n}$ let $M_P$ be a $2k$-dimensional Riemannian manifold constructed from $P\in\D_k$ by the Delzant construction. 

\begin{lemma}\label{extension}
For each $[g(\iota_k(P))]\in\tilde\D_{\leq n}$ the equivalence class of $M_P$ in ${\mathcal M}$ does not depend on a choice of the representative $P\in\D_k$. 
\end{lemma}
\begin{proof}
We show that if $\iota_k(P_1)=g(\iota_k(P_2))\in \D_{\leq n}$ for $g=(A, t)\in G_n$ and $P_1, P_2\in \D_k$, then $P_1=g'P_2$ for some $g'\in G_k$. 
Since $\iota_k(P_2)$ contains interior points in $\R^k(\subset \R^n)$ one has that $g(\R^k)=\R^k$, and hence, we have a block decomposition of $A$ and $t$ such as 
$$
A=\begin{pmatrix}
A_{11} & A_{12} \\ 
0 & A_{22}
\end{pmatrix}\in{\rm GL}(n,\Z), \quad 
t=\begin{pmatrix}
t' \\ 0 
\end{pmatrix}\in\R^k\oplus\R^{n-k}=\R^n
$$for some $A_{11}\in {\rm GL}(k,\Z), A_{22}\in{\rm GL}(n-k,\Z)$ and $t'\in \R^k$. It implies that $P_1=g'P_2$ for $g':=(A_{11}, t')\in{\rm AGL}(k,\Z)$. 
\end{proof}

Due to Lemma~\ref{extension} the following holds.

\begin{prop}
The map ${\rm Del} : \tilde\D_n\to {\mathcal M}$ can be extended to a map ${\rm Del}:\tilde\D_{\leq n}\to {\mathcal M}$, $[g(\iota_k(P))]\mapsto M_P$.  
\end{prop}

Since the metric $d$ on $\D_n$ cannot be extended to $\D_{\leq n}$ one may hope  to construct an alternative metric on $\D_{\leq n}$ and $\tilde\D_{\leq n}$. One can consider the {\it Hausdorff metric $d_H$} on $\D_{\leq n}$ as a candidate, however, the function $\tilde d_H: \tilde\D_{\leq n}\times \tilde\D_{\leq n}\to \R$ defined by 
\[
\tilde d_H(\alpha, \beta):=\inf\{d_H(P_1, P_2) \ | \ [P_1]=\alpha, [P_2]=\beta\}
\]for $\alpha, \beta\in \tilde\D_{\leq n}$ may degenerate. See the following example\footnote{This example is suggested by Yu Kitabeppu.}. 

\begin{example}
For $A:=\begin{pmatrix}1 & 2 \\ 2 & 5 \end{pmatrix}\in G_2$ let $P_1$ be a rectangle whose directional vectors are given by eigenvectors $u_1=\begin{pmatrix} 1 \\ 1+\sqrt{2} \end{pmatrix}$ and $u_2=\begin{pmatrix} 1 \\ 1-\sqrt{2} \end{pmatrix}$ of $A$ corresponding to eigenvalues $3+\sqrt{2}$ and $3-\sqrt{2}$. On the other hand let $P_2$ be a pentagon which is obtained by adding one vertex to $P_1$. See Figure~\ref{p1p2}.  
\begin{figure}
\begin{center}
\includegraphics[scale=0.65]{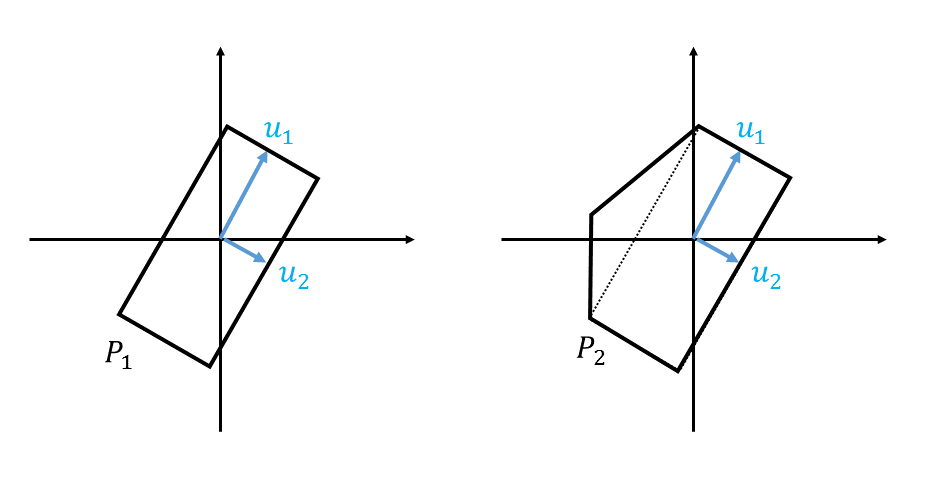}
\caption{A rectangle $P_1$ and a pentagon $P_2$} \label{p1p2}
\end{center}
\end{figure}
For this $A$, $P_1$ and $P_2$, we have that $P_1$ is not $G_n$-congruence to $P_2$ and 
\[
d_H(A^m(P_1), A^m(P_2))\to 0 \quad (m\to \infty). 
\] If we take a rectangle $P_1'$ (resp. pentagon $P_2'$) $\in \D_2$ which is close enough to $P_1$ (resp. $P_2$), then we have 
\[
\tilde d_H([P_1'], [P_2'])=0
\]  and $[P_1']\neq [P_2']$ in $\D_{\leq 2}$. 
\end{example}

In the subsequent paper \cite{FK}, to consider these problems,  we would like to use the {\it Wasserstein distance} between probability measures defined by convex bodies. 

\vspace{0.5cm}

\noindent{\bf Acknowledgements.}
The authors would like to thank Takashi Sato for giving a key idea of our proof of Proposition~\ref{positiveness} and Yu Kitabeppu for discussing about problems in Section~\ref{Relation with Gromov-Hausdorff distance}. They are grateful to Daisuke Kazukawa for telling us general properties on metric on the quotient space, Shouhei Honda for telling us recent developments in intrinsic flat metric. The authors also thank Yasufumi Nitta and Takahiko Yoshida for discussing on metrics on symplectic toric manifolds. 

\bibliography{reference}
\end{document}